\documentclass[12pt]{amsart}
\usepackage{graphicx}
\usepackage[active]{srcltx}
\usepackage{amsthm,mathrsfs}
\usepackage{a4wide}
\usepackage{amsfonts}
\usepackage{amsmath}
\usepackage{amssymb}

\newtheorem{theorem}{Theorem}
\newtheorem{corollary}{Corollary}

\usepackage{float}
\usepackage{url}
\usepackage{enumitem}

\makeatletter\def\blfootnote{\xdef\@thefnmark{}\@footnotetext}\makeatother
\allowdisplaybreaks
\title[Additive Energy and Irregularities of Distribution]{Additive Energy and Irregularities of Distribution}

\author{Christoph Aistleitner} 
\address{ Institute of Financial Mathematics and applied Number Theory, University Linz}
\email{christoph.aistleitner@jku.at}

\author{Gerhard Larcher} 
\address{ Institute of Financial Mathematics and applied Number Theory, University Linz}
\email{gerhard.larcher@jku.at}

\thanks{The first author is supported by the Austrian Science Fund (FWF) project I1751-N26. Both authors are supported by the FWF project F5507-N26, which is part of the Special Research Program ``Quasi-Monte Carlo Methods: Theory and Applications''}

\begin{document}

\begin{abstract}
We consider strictly increasing sequences $\left(a_{n}\right)_{n \geq 1}$ of integers and sequences of fractional parts $\left(\left\{a_{n} \alpha\right\}\right)_{n \geq 1}$ where $\alpha \in \mathbb{R}$. We show that a small additive energy of $\left(a_{n}\right)_{n \geq 1}$ implies that for almost all $\alpha$ the sequence $\left(\left\{a_{n} \alpha\right\}\right)_{n \geq 1}$ has large discrepancy. We prove a general result, provide various examples, and show that the converse assertion is not necessarily true.
\end{abstract}

\date{}
\maketitle

\section{Introduction and statement of results} \label{sect_1}

Let $\left(a_{n}\right)_{n \geq 1}$ be a strictly increasing sequence of positive integers. We are interested in distribution properties of the sequence $\left(\left\{a_{n} \alpha \right\}\right)_{n \geq 1}$, where $\alpha$ is a given real and $\left\{x\right\}$ denotes the fractional part of $x$. In particular we are interested in the behavior of the star-discrepancy $D_{N}^{*}$ of these sequences from a metrical point of view. The star-discrepancy $D_{N}^{*}$ of the first $N$ elements of a sequence $\left(x_{n}\right)_{n \geq 1}$ in $\left[\left. 0,1\right.\right)$ is defined by
$$
D_{N}^{*}(x_1, \dots, x_N) := \underset{0 < \beta \leq 1}{\sup} \left|\frac{A_{N} \left(\beta\right)}{N} - \beta \right|,
$$
where $A_{N}\left(\beta\right):= \# \left\{1 \leq n \leq N \left| \right. x_{n} \in \left[\left. 0,\beta\right.\right)\right\}$.
The sequence $\left(x_{n}\right)_{n \geq 1}$ is uniformly distributed in $\left[0,1\right)$ if and only if $\underset{N \rightarrow \infty}{\lim} D_{N}^{*} = 0$.\\

There exists a vast literature on the discrepancy of sequences of the form $\left(\left\{a_{n} \alpha\right\}\right)_{n \geq 1}$. The most basic and classical example for such a class of sequences are the Kronecker sequences $\left(\left\{n \alpha \right\}\right)_{n \geq 1}$. It is well known that for the discrepancy of the Kronecker sequence for almost all $\alpha$ we have
$$
ND_{N}^{*} = \mathcal{O}\left(\left(\log N\right)^{1+ \varepsilon}\right)
$$
for all $ \varepsilon > 0$, which is close to optimality since by a classical result of W.M. Schmidt the discrepancy of every infinite sequence satisfies
$$
ND_{N}^{*} = \Omega\left(\log N\right).
$$
For sequences of the form $\left(\left\{a_{n} \alpha \right\}\right)_{n \geq 1}$, R. C. Baker~\cite{Baker} has shown the following general metric result. Let $\left(a_{n}\right)_{n \geq 0}$ be a strictly increasing sequence of integers. Then for almost all $\alpha$ for the discrepancy of $\left(\left\{a_{n} \alpha\right\}\right)_{n \geq 1}$ we have 
$$
ND_{N}^{*} = \mathcal{O} \left(N^{\frac{1}{2}} \left(\log N\right)^{\frac{3}{2}+\varepsilon}\right)
$$
for all $\varepsilon > 0$. This result is best possible, up to logarithmic factors, since for example for lacunary sequences $\left(a_{n}\right)_{n \geq 1}$, that is for sequences satisfying $\frac{a_{n+1}}{a_{n}} \geq 1 + \delta$ with a $\delta > 0$, we have 
$$
ND_{N}^{*} = \Omega \left(N^{\frac{1}{2}}\right).
$$
Indeed in this case much sharper results are known -- see for example~\cite{Fuku} or~\cite{Philipp}. For general sequences $\left(a_{n}\right)_{n \geq 1}$ which grow neither linearly nor exponentially it is usually very hard to find the correct metric order of the discrepancy of $\left(\left\{a_{n} \alpha\right\}\right)_{n \geq 1}$, and in particular only very few metric lower bounds are known. There are some notable exceptions, such as for example~\cite{bp} and~\cite{bpt}, but they usually are either very restrictive or depend on a strong arithmetic structure and deep number-theoretic tools. In~\cite{ALMathe} the authors of the present paper developed a new, fairly general method by which one can obtain metric lower bounds for the discrepancy of sequences of the form $\left(\left\{a_{n} \alpha\right\}\right)_{n \geq 1}$. Amongst other results, for example the following was shown there (Corollary 1 in~\cite{ALMathe}). Let $P \in \mathbb{Z} \left[x\right]$ be a polynomial of degree $d \geq 2$. Then for the discrepancy of the sequence $\left(\left\{P(n) \alpha\right\}\right)_{n \geq 1}$ we 
have for almost all $\alpha$
$$
ND_{N}^{*} = \Omega \left(N^{\frac{1}{2}-\varepsilon}\right)
$$
for all $\varepsilon > 0$. Together with the general upper bounds of Baker this means that for these sequences we have the essentially largest possible metric order of $D_{N}^{*}$, namely $ND_{N}^{*} \approx N^{\frac{1}{2}}$. At this point one might assume that for any choice of $\left(a_{n}\right)_{n \geq 1}$ the discrepancy of $\left(\left\{a_{n} \alpha\right\}\right)_{n \geq 1}$ for almost all $\alpha$ either satisfies $ND_{N}^{*} = \mathcal{O} \left(N^{\varepsilon}\right)$ or $ND_{N}^{*} = \Omega \left(N^{\frac{1}{2}-\varepsilon}\right)$. This, however, is not true: in~\cite{ALMonats} it was shown that any asymptotic order for $ND_{N}^{*}$ between $N^{\varepsilon}$ and $N^{\frac{1}{2}}$ is possible for almost all $\alpha$. More precisely, Theorem~1 in~\cite{ALMonats} states the following. Let $0 < \gamma \leq \frac{1}{2}$ be given. Then there exists a strictly increasing sequence $\left(a_{n}\right)_{n \geq 1}$ of positive integers such that for the discrepancy of $\left(\left\{a_{n} \alpha\right\}\right)_{n \geq 
1}$ for almost all $\alpha$ we have $ND_{N}^{*} = \mathcal{O} \left(N^{\gamma}\right)$ and $ND_{N}^{*} = \Omega \left(N^{\gamma- \varepsilon}\right)$ for all $\varepsilon > 0$. An even more precise result has been recently obtained by Berkes, Fukuyama and Nishimura~\cite{fukuni}, by using a randomization technique.\\

The purpose of the present paper is to point out that the so-called additive energy of the sequence $\left(a_{n}\right)_{n \geq 1}$ of integers can give some information on the metric distribution behavior of $\left(\left\{a_{n} \alpha \right\}\right)_{n \geq 1}$. More precisely, we will show that an upper bound for the additive energy of $(a_n)_{n \geq 1}$ implies a lower bound for the metric discrepancy of $\left(\left\{a_{n} \alpha \right\}\right)_{n \geq 1}$, under an additional, relatively moderate, growth assumption on $(a_n)_{n \geq 1}$.  The additive energy $E(A)$ of finite sets $A$ of integers was studied very intensively in recent years, especially in connection with additive combinatorics (see for example~\cite{Tao}). For many classes of sequences, good upper bounds for the additive energy are known. Thus the link between additive energy and metric discrepancy allows us to identify many concrete interesting classes of integer sequences $\left(a_{n}\right)_{n \geq 1}$ for which we can give good lower bounds for the metric discrepancy of $\left(\left\{a_{n} \alpha \right\}\right)_{n \geq 1}$.\\

Let $A = \left\{b_{1}, \ldots, b_{N}\right\}$ be a finite sequence of integers. Then the additive energy $E(A)$ is defined by 
$$E(A) := \# \left\{\left(\left. x_{1}, x_{2}, x_{3}, x_{4}\right) \in A^{4} \right| x_{1} - x_{2} = x_{3} - x_{4}\right\},$$
i.e., $E(A)$ is the number of solutions of the equation $x_{1} - x_{2} = x_{3}-x_{4}$ with $x_{1}, x_{2}, x_{3}, x_{4} \in A$. If the elements $b_{1}, \ldots, b_{N}$ of $A$ are pairwise distinct then obviously we always have $N^{2} \leq E(A) \leq N^{3}$. In the present paper we will prove the following theorem, which allows us to deduce lower bounds for the metric discrepancy from upper bounds for the additive energy.

\begin{theorem} \label{th_a}
Let $\left(a_{n}\right)_{n \geq 1}$ be a strictly increasing sequence of positive integers with $a_{n} \ll e^{\gamma \left(\log n\right)^{2}}$ for some $\gamma > 0$. Assume furthermore that $E\left(\left\{a_{1}, \ldots, a_{N}\right\}\right) \ll N^{\kappa}$ for some $\kappa \in [2,3]$. Then for the discrepancy of the sequence $\left(\left\{a_{n} \alpha \right\}\right)_{n \geq 1}$ for almost all $\alpha$ we have
$$
ND_{N}^{*} = \Omega \left(N^{\frac{3-\kappa}{2}-\varepsilon}\right)
$$
for all $\varepsilon > 0$.
\end{theorem}

Our result shows in particular that for $\left(a_{n}\right)_{n \geq 1}$ with smallest possible additive energy, i.e. in the case $\kappa=2$, the sequence $\left(\left\{a_{n} \alpha\right\}\right)_{n \geq 1}$ for almost all $\alpha$ has essentially the largest possible discrepancy. More precisely, the following corollary holds.

\begin{corollary} \label{co_a}
Let $\left(a_{n}\right)_{n \geq 1}$ be a strictly increasing sequence of positive integers for which $a_{n} \ll e^{\gamma \left(\log N\right)^{2}}$ for some $\gamma >0$ and
$$
E\left(\left\{a_{1}, \ldots, a_{N}\right\}\right) = \mathcal{O} \left(N^{2+ \varepsilon}\right) \qquad \text{for all} \qquad \varepsilon > 0.
$$
Then for the discrepancy of $\left(\left\{a_{n} \alpha \right\}\right)_{n \geq 1}$ for almost all $\alpha$ we have $ND_{N}^{*} = \mathcal {O} \left(N^{\frac{1}{2} + \varepsilon}\right)$ and $ND_{N}^{*} = \Omega \left(N^{\frac{1}{2} -\varepsilon}\right)$ for all $\varepsilon > 0$.
\end{corollary}

Our proof of Theorem~\ref{th_a} is essentially just a slight extension of the proof of Theorem~3 in~\cite{ALMathe}. However, as already noted, the fact that Theorem~\ref{th_a} in the present paper is formulated in the language of additive combinatorics allows us to use several pre-existing results on the additive energy of integer sequences to obtain lower metric bounds for the discrepancy of $\left(\left\{a_{n} \alpha \right\}\right)_{n \geq 1}$ for several interesting specific examples of sequences $\left(a_{n}\right)_{n \geq 1}$. So for example we can deduce the following results. 

\begin{theorem} \label{th_b}
Let $\left(a_{n}\right)_{n \geq 1}$ be a sequence of integers with $a_{n} \ll e^{\gamma \left(\log n\right)^{2}}$ for some $\gamma >0$ which is \emph{convex}, i.e., which satisfies
$$
a_{n+1} -a_{n} > a_{n}  -a_{n-1} \qquad \textrm{for} \qquad n \geq 2.
$$
Then for almost all $\alpha$ for the discrepancy of the sequence $\left(\left\{a_{n} \alpha\right\}\right)_{n \geq 1}$ we have
$$
ND_{N}^{*} = \Omega \left(N^{\frac{7}{26}-\varepsilon}\right)
$$
for all $\varepsilon > 0$.
\end{theorem}

\begin{proof}[Proof of Theorem~\ref{th_b}]
The result follows immediately from Theorem~\ref{th_a} above and from Theorem~1 in~\cite{Shkredov}, which states that for every convex set $A$ of $N$ elements we have $E(A) \ll N^{\frac{32}{13}} \left(\log N\right)^{\frac{71}{65}}$.
\end{proof}

\begin{theorem} \label{th_c}
Let $a_{n} = \left\lfloor F(n)\right\rfloor$ where the real-valued function $F$ is three times continuously differentiable on $\left[\left.1,\infty\right.\right)$ and satisfies
$$
F'(x) > 0, \qquad F''(x) > 0, \qquad \text{and} \qquad  F'''(x) < 0
$$ 
on $\left[\left. 1, \infty\right.\right)$. Then for almost all $\alpha$ for the discrepancy of the sequence $\left(\left\{a_{n} \alpha\right\}\right)_{n \geq 1}$ we have
$$
ND_{N}^{*} = \Omega \left(N^{\min \left(\frac{1}{4}, \frac{1+\rho_{F}}{2}\right)-\varepsilon}\right)
$$
for all $\varepsilon > 0$, where $\rho_{F} := \underset{N \rightarrow \infty}{\lim \inf}~\frac{\log F''(N)}{\log N}$.
\end{theorem}

\begin{proof}[Proof of Theorem~\ref{th_c}]
The result follows immediately from Theorem~\ref{th_a} above and from Corollary~2 in~\cite{Garaev} which shows that
$$
E\left(\left\{a_{1}, \ldots, a_{N}\right\}\right) \ll N^{\frac{5}{2}} + \frac{N^{2} \log N}{F''(N)}
$$
for all $N$.
\end{proof}

From Theorems~\ref{th_b} and~\ref{th_c} we immediately obtain the following examples:

\begin{corollary} \label{co_b}
Let $a_{n} := \left\lfloor n^{c}\right\rfloor$ for some $c \in \left(1, \infty\right)$. Then for almost all $\alpha$ for the discrepancy of the sequence $\left(\left\{a_{n} \alpha\right\}\right)_{n \geq 1}$ we have
$$
ND_{N}^{*} = \Omega \left(N^{\tau-\varepsilon}\right)
$$
for all $\varepsilon > 0$, where

\[\tau=\begin{cases}
\frac{c-1}{2} & \text{if} \quad 1 < c < \frac{3}{2} \\
\frac{1}{4} & \text{if} \quad \frac{3}{2} \leq c < 2 \\
\frac{7}{26} &\text{if} \quad c \geq 2.
\end{cases}\]
\end{corollary}

\begin{corollary} \label{co_c}
Let $a_{n} := \left\lfloor e^{\gamma \left(\log N \right)^{\beta}}\right\rfloor$ for some $\gamma >0$ and $\beta$ with $1 < \beta \leq 2$. Then for almost all $\alpha$ for the discrepancy of $\left(\left\{a_{n} \alpha\right\}\right)_{n \geq 1}$ we have
$$
ND_{N}^{*} = \Omega \left(N^{\frac{7}{26}-\varepsilon}\right)
$$
for all $\varepsilon > 0$.
\end{corollary}

The following result was already given in Theorem~1 in \cite{ALMathe}. Since we found that there is an inaccuracy in the proof of Theorem~1 in \cite{ALMathe}, we formulate this result here again as Corollary~\ref{co_d}, and we give an alternative proof.\\

\begin{corollary} \label{co_d}
Let $P \in \mathbb{Z} \left[x\right]$ be a polynomial of degree $d \geq 2$ and
let $\left(m_{n}\right)_{n \geq 1}$ be an arbitrary sequence of pairwise
different integers with $\left|m_{n}\right| \leq n^{t}$ for some $t \in
\mathbb{N}$ and all $n \geq n(t)$. Then for the discrepancy $D_{N}$ of the
sequence $\left(\left\{P\left(m_{n}\right)\alpha \right\}\right)_{n \geq 1}$ we
have for almost all $\alpha$
$$
ND_{N} \geq N^{\frac{1}{2}-\varepsilon}
$$
for all $\varepsilon >0$ and for infinitely many $N$.
\end{corollary}

\begin{proof}[Proof of Corollary~\ref{co_d}]
Let $f(n):= P\left(m_{n}\right)$. For sufficiently large $N$ we have 
\begin{eqnarray*}
E\left(\left\{f(n) \left|\right.n=1, \ldots, N\right\}\right) & = & \sum_{\underset{f(k) -f(l) = f(m)-f(n)}{1 \leq k, l, m, n \leq N}} \\
& \leq & \sum_{a} \left(\tilde{A}_{f} (a)\right)^2.
\end{eqnarray*}
Here $\tilde{A}_{f} (a) := \left\{\left(x,y\right) \in \mathbb{N} \times \mathbb{N} \left| \right. f(x) -f(y) = a\right\}$, and the summation in the last sum is extended over all integers $a \in [-c_{P} N^{dt},c_{P} N^{dt}]$ such that there exist $k,l$ with $1 \leq k, l\leq N$ and $f(k) -f(l) =a$, where $c_{P}$ is an appropriate constant depending only on $P$. Note that the number of such $a$ is at most $N^{2}$. Hence by Theorem~\ref{th_a} it suffices to show that $\tilde{A}_{f} (a) = \mathcal{O} \left(\left|a\right|^{\varepsilon}\right)$ for all $\varepsilon > 0$.\\

Let us first assume that $m_{n} =n$. Since $f$ is of degree $d \geq 2$, there exists a non-constant $q \in \mathbb{Z}\left[x,y\right]$ such that $f(x) -f(y) = (x-y) q(x,y)$. So, if $f(x) -f(y) = a$ for some non-zero integer $a$, it follows that $x-y$ is a divisor $t$ of $a$, and that hence $q\left(x, x-t\right) = \frac{a}{t}$. This last equation has at most $d-1$ solutions $x$. Since $a$ has $\mathcal{O}\left(\left|a\right|^{\varepsilon}\right)$ divisors $t$ for all $\varepsilon > 0$, the assertion $\tilde{A}_{f} (a) = \mathcal{O}\left(\left|a\right|^{\varepsilon}\right)$ follows.\\

For arbitrary $\left(m_{n}\right)$ the result follows trivially from the special
result for $m_{n} =n$.
\end{proof}

As already mentioned above, for the additive energy $E(A)$ of a finite set $A$ of distinct integers we always have $\left|A\right|^{2} \leq E(A) \leq \left|A\right|^{3}$, and for every strictly increasing sequence $\left(a_{n}\right)_{n \geq 1}$ of positive integers for almost all $\alpha$ the order of the discrepancy of $\left(\left\{a_{n} \alpha\right\}\right)_{n \geq 1}$ essentially is between $N^{\varepsilon}$ and $N^{\frac{1}{2}}$. The quintessence of Theorem~\ref{th_a} is that a small order of the additive energy of $\left\{a_{1}, \ldots, a_{N}\right\}$ for all $N$ implies a large metric order of $ND_{N}^{*}$ for $\left(\left\{a_{n} \alpha\right\}\right)_{n \geq 1}$. In particular, the lowest possible order of the additive energy of $\left\{a_{1}, \ldots, a_{N}\right\}$ for all $N$ implies the largest possible metric order of $ND_{N}^{*}$ of $\left(\left\{a_{n} \alpha\right\}\right)_{n \geq 1}$. \\

It is tempting to ask whether the converse statement also is true, that is whether a large order of the additive energy of $\left\{a_{1}, \ldots, a_{N}\right\}$ for all $N$ necessarily implies a small metric order of $ND_{N}^{*}$ of $\left(\left\{a_{n} \alpha\right\}\right)_{n \geq 1}$. This hypothesis seems to be supported by the pure Kronecker sequence $\left(\left\{n \alpha\right\}\right)_{n \geq 1}$, i.e., $a_{n} = n$. The additive energy in this case satisfies $E\left(\left\{a_{1}, \ldots, a_{N}\right\}\right) \gg N^{3}$, so it is of the largest possible order, and $ND_{N}^{*} = \mathcal{O} \left(N^{\varepsilon}\right)$ for all $\varepsilon >0$ for almost all $\alpha$, which means that the discrepancy is of the lowest possible order.\\

However, the hypothesis is not true, as the example given in Theorem~\ref{th_d} below shows. There we present a sequence $\left(a_{n}\right)_{n \geq 1}$ which has both the largest possible order of the additive energy as well as the largest possible metric order of the discrepancy for $\left(\left\{a_{n} \alpha\right\}\right)_{n \geq 1}$. This sequence $\left(a_{n}\right)_{n \geq 1}$ is characterized by the Rudin--Shapiro sequence. The Rudin--Shapiro sequence $\left(r_{0}, r_{1}, r_{2}, \ldots\right) = \left(1,1,1,-1,1,1,-1,1, \ldots\right)$ is defined by

\[r_{k}=\begin{cases}
1 & \text{if the number of}~11\text{-blocks in the base}~2~\text{representation of}~k~\text{is even}\\
-1 & \text{otherwise.} \\
\end{cases}\]

Let the sequence $\left(a_{n}\right)_{n \geq 1} = \left(0,1,2,4,5,7, \ldots\right)$ be the sequence of those indices $k$ for which $r_k=1$ in the Rudin--Shapiro sequence, sorted in increasing order. We will call this sequence the sequence of \emph{Rudin--Shapiro integers}. By construction this sequence is strictly increasing and we have $a_{n} \leq 2 n$ for all $n$.

\begin{theorem} \label{th_d}
Let $\left(a_{n}\right)_{n \geq 1}$ be the sequence of Rudin--Shapiro integers. Then the additive energy of $\left\{a_{1}, \ldots ,a_{N}\right\}$ is of maximal possible order, i.e., 
$$
E\left(\left\{a_{1}, \ldots, a_{N}\right\}\right) \gg N^{3},
$$
and for almost all $\alpha$ the discrepancy of $\left(\left\{a_{n} \alpha\right\}\right)_{n \geq 1}$ is also essentially of maximal possible order, i.e.,
$$
ND_{N}^{*} = \Omega \left(N^{\frac{1}{2}-\varepsilon}\right)
$$
for all $\varepsilon > 0$.
\end{theorem}

This result should be compared to another example which was given in~\cite{AHL}. There the sequence $\left(a_{n}\right)_{n \geq 1}$ of Thue--Morse integers (also called \emph{evil numbers}) was studied. This sequence also has additive energy of maximal possible order, and in this case for almost all $\alpha$ for the discrepancy of $\left(\left\{a_{n}\alpha\right\}\right)_{n\geq1}$ we have $ND_{N}^{*} = \mathcal{O} \left(N^{0.4035}\right)$ and $ND_{N}^{*} = \Omega \left(N^{0.4033}\right)$.\\

Further examples of sequences $\left(a_{n}\right)_{n \geq 1}$ with highest possible additive energy and the essentially maximal metric order of $ND_{N}^{*}$ of $\left(\left\{a_{n} \alpha \right\}\right)_{n\geq1}$ can be deduced from the results in \cite{AF1,AF2,fukuni}, where it may be necessary to modify the sequences constructed there by inserting long stretches of arithmetic progressions in order to maximize the additive energy. However, the examples given in these papers are randomly generated sequences and no explicit constructions are known, whereas the example in Theorem~\ref{th_d} above is fully explicit.\\

It remains to prove Theorems~\ref{th_a} and~\ref{th_d}. These proofs will be given in the next section.

\section{The proofs of Theorem~\ref{th_a} and of Theorem~\ref{th_d}} \label{sect_2}

\begin{proof}[Proof of Theorem~\ref{th_a}]
For a strictly increasing sequence $\left(a_{n}\right)_{n \geq 1}$ of positive integers we set $I(N) := \int^{1}_{0} \left|\sum^{N}_{n=1} e^{2 \pi i a_{n} \alpha}\right| d \alpha$, and we write $E\left(A_{N}\right)$ for the additive energy of $A_{N} := \left\{a_{1}, \ldots, a_{N}\right\}$. Note that by orthogonality we have $E(A_N) = \int^{1}_{0} \left|\sum^{N}_{n=1} e^{2 \pi i a_{n} \alpha}\right|^4 d \alpha$. By a classical trick, which is based on a clever application of H\"older's inequality, we have $I(N) \geq \left(\frac{N^{3}}{E\left(A_{N}\right)}\right)^{\frac{1}{2}}$ (see for example~\cite[Theorem 1]{Kar}). Hence, if $E\left(A_{N}\right) \ll N^{\kappa},$ then
\begin{equation} \label{equ_a}
I(N) \gg N^{\frac{3-\kappa}{2}}.
\end{equation}
In~\cite[Theorem~3]{ALMathe} the following result~(*) was shown:\\

\textit{Let $\left(a_{n}\right)_{n \geq 1}$ be a sequence of integers such that for some $t \in \mathbb{N}$ we have $\left|a_{n}\right| \leq n^{t}$ for all $n$ large enough. Assume there exist a number $\tau \in \left(0,1\right)$ and a strictly increasing sequence $\left(B_{L}\right)_{L \geq 1}$ of positive integers with $\left(B'\right)^{L} \leq B_{L} \leq B^{L}$ for some reals $B', B$ with $1 < B' < B$, such that for all $\varepsilon >0$ and all $L > L(\varepsilon)$ we have $I\left(B_{L}\right) > B^{\tau-\varepsilon}_{L}$. Then for almost all $\alpha \in \left[\left. 0,1\right.\right)$ for all $\varepsilon >0$ for the discrepancy $D_{N}^{*}$ of the sequence $\left(\left\{a_{n} \alpha\right\}\right)_{n \geq 1}$ we have $ND_{N}^{*} = \Omega \left(N^{\tau-\varepsilon}\right)$.}\\

The proof of this result was based on a further result~(**), which is stated as Theorem~4 in~\cite{ALMathe}:\\

\textit{Let $\left(R_{L}\right)_{L \geq 0}$ be a sequence of measurable subsets of $\left[\left. 0,1\right.\right)$, with the measure $\mathbb{P}\left(R_{L}\right)$ of $R_{L}$ satisfying $\mathbb{P}\left(R_{L}\right) \geq \frac{1}{B^{L}}$ for some constant $B>0$, and such that each $R_{L}$ is the disjoint union of at most $A^{L}$ intervals for some $A>0$. Then for almost all $\alpha \in \left[\left. 0,1\right.\right)$ for every $\eta >0$ there are infinitely many integers $h_{L}$ with $h_{L} \leq \left(1+\eta\right)^{L} \frac{1}{\mathbb{P} \left(R_{L}\right)}$ and $\left\{h_{L} \alpha\right\} \in R_{L}$.}\\

It turns out that this last result~(**) is also true in a stronger version, namely under the weaker assumption that $R_{L}$ is the disjoint union of at most $A^{L^{2}}$ intervals for some $A > 0$. This can easily be seen by following the proof of Theorem~\ref{th_d} in~\cite{ALMathe} line by line, replacing $A^{L}$ by $A^{L^{2}}$ and choosing the value $G_{L}$ which appears in the proof as $G_{L} := A^{2 L^{2}} \left(B\left(1+ \eta\right)\right)^{2 L}$ instead of $G_{L} = \left(AB \left(1+ \eta\right)\right)^{2 L}$.\\

From this a stronger version of~(*) follows, namely the fact that the conclusion of (*) also holds under the weaker assumption that $a_{n} < e^{\gamma\left(\log n\right)^{2}}$ for some $\gamma > 0$. This can also be easily seen by following the proof of Theorem~\ref{th_c} in~\cite{ALMathe} line by line. We just have to change formula (18) in this proof in~\cite{ALMathe} to
\begin{eqnarray*}
\left|f_{L} \left(\alpha_{1}\right) - f_{L} \left(\alpha_{2}\right)\right| & \leq & 2 \pi B_{L} e^{\gamma \left(\log B_{L}\right)^{2}}\\
& \leq & 2 \pi B^{L} e^{\gamma L^{2} \left(\log B\right)^{2}} \\
& \ll & A^{L^{2}}
\end{eqnarray*}
for some constant $A > 1$. As a consequence the function $g_{L}$ appearing in the proof can be written as a sum of $\ll A^{L^{2}}$ indicator functions of disjoint intervals, and hence the set $M_{L}^{\left(i_{L}\right)}$ appearing in the proof is always a union of $\ll A^{L^{2}}$ intervals. Then the stronger version of~(**) which we have obtained above is used to establish the stronger version of~(*). The desired result then follows immediately from this stronger version of~(*) together with~\eqref{equ_a}.
\end{proof}

\begin{proof}[Proof of Theorem~\ref{th_d}]
Let $\rho_{n} (x) := \sum^{2^{n}-1}_{k=0} r_{k} x^{k}$ be the Rudin--Shapiro polynomials. From Theorem 2.1 in~\cite{Erd.} it follows that $\int^{1}_{0} \left|\rho_{n} \left(e^{2 \pi i \alpha}\right)\right| d\alpha \gg 2^{\frac{n}{2}}$ for all $n$. Let $\sum (n)~:=~\sum^{\sigma(n)}_{k=0} e^{2 \pi i a_{k} \alpha}$, where $\sigma(n) := \#\left\{0 \leq k < 2^{n} \left| \right. r_{k} = 1\right\}.$ Then
\begin{equation} \label{equ_b}
\sum (n) = \frac{1}{2} \left(\rho_{n} \left(e^{2 \pi i \alpha}\right)+ \sum^{2^{n}-1}_{k=0} e^{2 \pi i k \alpha}\right).
\end{equation}
We have
\begin{eqnarray*}
\int^{1}_{0} \left|\sum^{2^{n}-1}_{k=0} e^{2 \pi i k \alpha}\right| d \alpha & \leq & \int^{1}_{0} \min \left(2^{n}, \frac{1}{\left\|\alpha\right\|}\right) d \alpha \\
& \leq & 2 + 2 \int^{\frac{1}{2}}_{\frac{1}{2^{n}}} \frac{1}{\alpha} d \alpha \\
& \leq & 2 + 2n.
\end{eqnarray*}
Hence
\begin{equation} \label{equ_c}
\int^{1}_{0} \left|\sum(n)\right|d \alpha \gg \int^{1}_{0} \left|\rho_{n} \left(e^{2 \pi i \alpha}\right)\right| d\alpha - \int^{1}_{0} \left|\sum^{2^{n}-1}_{k=0} e^{2 \pi i k \alpha}\right| d \alpha \gg 2^{\frac{n}{2}}.
\end{equation}
It is also well known (see for example~\cite{Br+Mo}) that we always have $\sqrt{\frac{3l}{5}} < \sum^{l-1}_{k=0} r_{k} < \sqrt{6 l}$ and therefore
\begin{equation} \label{equ_d}
a_{k} \leq 2k
\end{equation}
for all $k$. From~\eqref{equ_d} we immediately obtain $E\left(\left\{a_{1}, \ldots, a_{N}\right\}\right) \gg N^{3}$, and from~\eqref{equ_c} and~\eqref{equ_d} and using Theorem~\ref{th_c} in~\cite{ALMathe} we obtain the desired $\Omega$-estimate for $ND_{N}^{*}$.
\end{proof}

\end{document}